\documentclass[12pt]{amsart}
\usepackage{amsmath,amsthm,amsfonts,amssymb,times,latexsym,mathabx,url,bbm}

\newtheorem*{theorem*}{Theorem}
\newtheorem{theorem}{Theorem}

\newtheorem{lem}{Lemma}[section]

\numberwithin{equation}{section}

\DeclareMathOperator{\AP}{AP}

\makeatletter
\renewcommand{\pmod}[1]{\allowbreak\mkern7mu({\operator@font mod}\,\,#1)}
\makeatother

\newcommand{\A}{\mathcal{A}}

\renewcommand{\d}{\delta}
\newcommand{\eps}{\varepsilon}
\newcommand{\D}{\mathcal D}

\renewcommand{\l}{\lambda}

\newcommand{\N}{\mathbb{N}}
\renewcommand{\o}{\omega}

\renewcommand{\P}{\mathbb P}

\newcommand{\s}{\sigma}

\newcommand{\U}{\mathcal U}
\newcommand{\V}{\mathcal V}
\newcommand{\Z}{\mathbb{Z}}
\newcommand{\HH}{\mathfrak H}

\renewcommand{\b}{\mathbf b}
\newcommand{\e}{\mathbf e}
\newcommand{\n}{\mathbf n}
\renewcommand{\S}{\mathbf S}
\newcommand{\bAP}{\mathbf{AP}}
\newcommand{\bl}{\boldsymbol{\lambda}}
\newcommand{\Q}{\mathcal Q}

\newcommand{\bcR}{\boldsymbol{\mathcal{R}}}
\newcommand{\mE}{\mathcal E}
\newcommand{\bmE}{\boldsymbol{\mathcal E}}
\newcommand{\mF}{\mathcal F}
\newcommand{\bmF}{\boldsymbol{\mathcal F}}

\renewcommand{\leq}{\leqslant}
\renewcommand{\geq}{\geqslant}

\newcommand{\NN}{{\mathbb N}}

\newcommand{\cA}{\ensuremath{\mathcal{A}}}

\newcommand{\cM}{\ensuremath{\mathcal{M}}}
\newcommand{\cN}{\ensuremath{\mathcal{N}}}

\newcommand{\cQ}{\ensuremath{\mathcal{Q}}}
\newcommand{\cR}{\ensuremath{\mathcal{R}}}

\newcommand{\br}{\ensuremath{\mathbf{r}}}

\newcommand{\bu}{\ensuremath{\mathbf{u}}}

\newcommand{\E}{\mathbb{E}}

\makeatletter
\renewcommand{\pmod}[1]{\allowbreak\mkern7mu({\operator@font mod}\,\,#1)}
\makeatother
\newcommand{\be}{\begin{equation}}
\newcommand{\ee}{\end{equation}}

\renewcommand{\ssum}[1]{\sum_{\substack{#1}}}  

\renewcommand{\le}{\leqslant}
\renewcommand{\leq}{\leqslant}
\renewcommand{\ge}{\geqslant}
\renewcommand{\geq}{\geqslant}

\newcommand{\fl}[1]{{\ensuremath{\left\lfloor {#1} \right\rfloor}}}  
\renewcommand{\(}{\left(}
\renewcommand{\)}{\right)}

\newcommand{\one}{\ensuremath{\mathbbm{1}}}  

\begin{document}
	
\title{Long strings of consecutive composite values of polynomials}
\author{Kevin Ford, Mikhail R. Gabdullin}
\date{}
\address{Department of mathematics, 1409 West Green Street, University of Illinois at Urbana-Champaign, Urbana, IL 61801, USA}
\email{ford@math.uiuc.edu}

\address{Department of mathematics, 1409 West Green Street, University of Illinois at Urbana-Champaign, Urbana, IL 61801, USA; Steklov Mathematical Institute,
Gubkina str., 8, Moscow, 119991, Russia}
\email{gabdullin.mikhail@yandex.ru, mikhailg@illinois.edu}

\begin{abstract}
We show that for any polynomial $f:\Z\to \Z$ with positive leading coefficient and
irreducible over $\mathbb{Q}$, if $x$ is large enough then there is a string of
$(\log x)(\log\log x)^{1/835}$ consecutive integers $n\in [1,x]$ for which
$f(n)$ is composite.  This improves the result in \cite{sieved}, which 
states that there are such strings of length $(\log x)(\log\log x)^{c_f}$,
where $c_f$ depends on $f$ and $c_f$ is exponentially
small in the degree of $f$ for some polynomials.  
\end{abstract}



\thanks{KF was supported by National Science Foundation Travel Support grant DMS-2301264
and Simons Foundation grant MP-TSM-00002373.}

\date{\today}
\maketitle	
	
\section{Introduction} 	

The first author, together with Konyagin, Maynard, Pomerance and Tao, showed in
 \cite{sieved} that a general ``sieved set'' contains long gaps.  More precisely,
 for each prime $p$ consider a set $I_p$ of residue classes modulo $p$,
 and call the collection of all sets $I_p$ a \emph{sieving system}.
As in \cite{sieved}, assume the following regularity conditions:
 \begin{itemize}
 \item[(a)] We have $|I_p|\le p-1$ for all $p$;
\item[(b)] $|I_p|$ is bounded; there is a $B\in \NN$ with $|I_p|\le B$ for all $p$;
\item[(c)] $|I_p|$ has average value 1, in the sense that
\be\label{Mertens}
\prod_{p\le x} \(1-\frac{|I_p|}{p}\) \sim \frac{C_1}{\log x}  \qquad (x\to\infty),
\ee
for some constant $C_1>0$. 
\item[(d)] There is a $\rho>0$, so that
the density of primes with $|I_p|\ge 1$ equals $\rho$, that is,
\[
\lim_{x\to\infty} \frac{|\{p\le x : |I_p|\ge 1\}|}{x/\log x}=\rho.
\]
\end{itemize}

Now define the \emph{sieved set}
\[
S_x:= \Z\setminus \bigcup_{p\le x} I_p,
\]
so that $S_x$ is a periodic set with period equal to the product of
the primes $p\le x$. The main theorem from \cite{sieved} states that if (a)--(d) hold, then for any $\eps>0$ and large $x$, the set $S_x$ contains a gap of size $x (\log x)^{C(\rho)-\eps}$, where

\begin{equation}\label{Crho-def} 
C(\rho) :=\sup\Big\{\d>0: \frac{6\cdot 10^{2\delta}}{\log(1/(2\d))}<\rho\Big\}.	
\end{equation}

In particular, $C(\rho)$ decays exponentially in $1/\rho$.
One of the principal applications of this result in \cite{sieved} is to finding
long strings of consecutive composite values of polynomial sequences. 
Consider a polynomial $f:\Z\to\Z$ of degree $d \geq 1$, with positive leading
coefficient and irreducible over $\mathbb{Q}$.
Let $I_p=\emptyset$ for $p\le d$ and 
\[
 I_p := \{ n \in \Z/p\Z: f(n)\equiv 0\pmod{p} \} \qquad (p>d).
 \]
The polynomial need not have integer coefficients.  Indeed, by P\'olya's theorem \cite{Polya},
$f$ is integer valued at integers if and only if $f$ 
has the form $f(x)=\sum_{j=0}^d a_j \binom{x}{j}$ with every $a_j\in \Z$.  In particular, $d! f(y) \in \Z[y]$ and
thus the sieving system is well-defined.
We call the collection of all sets $I_p$ a \emph{polynomial sieving system.}

By Lagrange's theorem, $|I_p| \le d < p $ for all $p>d$,
and hence (a) and (b) hold.  Item (c) holds 
by Landau's Prime
Ideal Theorem \cite{Landau} (see also \cite[pp. 35--36]{CoMu}), 
while (d) 
 follows from the Chebotarev Density
Theorem \cite{Chebotarev} (see also \cite{lagodl}),
with $\rho=\rho(f)$ equal to $C/|G|$, where $G$ is the Galois group of $f$,
a subgroup of the symmetric group on $d$ objects,
and $C$ is the number of elements of $G$ having at least one fixed point.
We have $\rho(f)\ge 1/d$ always, and sometimes $\rho(f)=1/d$.

\textbf{Examples.}  When $f(x)=x^2+1$, we have $|I_2|=1$, $|I_p|=2$ for all 
$p\equiv 1\pmod{4}$ and $|I_p|=0$ for $p\equiv 3\pmod4$.  Thus, $\rho(f)=1/2$.
Similarly, if $d=2^k$ for a positive integer $k$ and $f(x)=x^d+1$, then $|I_p|=d$ for
$p\equiv 1 \pmod{2^{k+1}}$ and $|I_p|=0$ for all other odd primes $p$,
hence $\rho(f)=1/d=1/2^k$.

The connection with sieved sets comes from the obvious relation
\begin{equation*}
 \{ n \in \N: f(n)>x, f(n) \text{ prime} \} \subset S_x.
\end{equation*}
Now let $X$ be large and set $x := \frac{1}{2} \log X$.  
The period of the set $S_x$ is $X^{1/2+o(1)}$ by the Prime Number Theorem.  Thus, 
writing $g_x$ for the longest gap in $S_x$, 
the interval $[X/2,X]$ contains a gap in $S_x$ of length $g_x$.
For large $X$, $f(n)>x$ for all $n\in (X/2,X]$, and we conclude that
\be\label{large-polynomial-gaps-from-Sx}
\max \{ m: X/2 < n < n+m \le X \text{ and }f(n),\ldots,f(n+m)\text{ composite} \} \ge g_{(1/2)\log X}.
\ee

The main theorem of \cite{sieved} implies that the left side
of \eqref{large-polynomial-gaps-from-Sx} is $\ge (\log X)(\log\log X)^{C(\rho(f))-o(1)}$.
In particular, we have \cite[Corollary 1]{sieved}, which states that
the left side of \eqref{large-polynomial-gaps-from-Sx}
is at least of size $\ge (\log X)(\log\log X)^{C(1/d)-o(1)}$.
The exponent here decays exponentially in $d$ (roughly $C(1/d) \approx e^{-6d}$).

It is still an open conjecture (of Bunyakovsky \cite{bun}) that there are infinitely many integers $n$ for which $f(n)$ is prime. Moreover it is believed (see the conjecture of Bateman and Horn \cite{BatemanHorn}) that the density of these prime values on $[X/2,X]$ is $\asymp_f 1/\log{X}$, and so the gaps above would be unusually large compared to the average gap of size $\asymp_f \log{X}$.

\bigskip

Our main result is a stronger lower bound on the length of strings of
consecutive composite values of polynomials, with exponent of $\log\log x$
being independent of $f$.

\begin{theorem}\label{thm:main}
Let $f$ be as above.  For all $\eps>0$ and large $X$, there are $n,n+m\in (X/2,X]$
with $m \ge (\log X)(\log\log X)^{C(1)-\eps}$ and with $f(n),\ldots,f(n+m)$
all composite.
\end{theorem}

Numerically, $1/C(1) = 834.109\ldots$.
In particular, when $f(x)=x$, our bound falls well short of the best known lower bound
for the maximal gap between the primes below $x$ from \cite{FGKMT}, which is
\[
\gg (\log X)(\log\log X) \frac{\log\log\log\log X}{\log\log\log X}.
\]
However, as noted in \cite{sieved}, the methods used to find large gaps between
 primes do not apply to gaps in more general sieved sets.

\bigskip 

Our proof is based on the method developed in \cite{sieved}, but with one important
difference.  In \cite{sieved}, only one of the elements of $I_q$ is utilized for
large $q$ (for those $q$ with $|I_q|\ge 1$), whereas in this paper we utilize all
of the set $I_q$.  This introduces a number of complications, which we get around using
special properties of polynomial sequences.  Our methods do not apply for all
of the sieved sets considered in \cite{sieved}, but they do generalize 
to sieved sets for which the sizes of the $I_p$
have a limiting distribution and for which
the difference sets $I_p-I_p:=\{a-a' : a,a'\in I_p\}$, interpreted as subsets of $\Z$, do not have large overlap.
To state our general theorem, we introduce further conditions, Hypotheses (e), (f) and (g) (here and throughout the paper, the symbols $p$ and $q$ always denote primes):
\begin{itemize}
\item[(e)] For each $\nu \in \{1,\ldots,B\}$, the density of those $p$ with $|I_p|=\nu$
exists.  That is, for some non-negative real numbers $\rho_\nu$, $1\le \nu\le B$,
we have
\[
\lim_{x\to \infty} \frac{\# \{ p\le x: |I_p|=\nu \}}{x/\log x} = \rho_\nu.
\]
\item[(f)] For non-zero $v$, define
\[
N(v) = \# \{ p : v\bmod p \in I_p-I_p \}.
\]
Then, for all $v\ge 1$, we have $N(v) \ll v^{0.49}$.

\item[(g)]
There are  positive constants $c_1,c_2$ such that the following holds.
Let $u\ge 10$, and for each prime $q$ with $|I_q| \ge 1$, let $m_q$ be a nonzero integer with $m_q\bmod q \in (I_q-I_q)$.  If $|w|\le u$ and $k\ge 1$, then
\[
\# \big\{ q : |I_q|\ge 1, 0 <|m_q| \le u, m_q+w\ne 0, N(m_q+w) \ge k \big\} \ll u (\log u)^{c_1} e^{-c_2 k}.
\]
\end{itemize}

Hypothesis (e) is stronger than Hypothesis (d) and will replace it. 
Furthermore, (e) implies that the average of $|I_p|$, over $p\le x$,
is asymptotically $\rho_1 + 2\rho_2 + \cdots + B\rho_B$, which, by the weak average
assumption (c), equals 1.

\bigskip 

\begin{theorem}\label{thm: general}
Consider any sieving system satisfying conditions (a)--(c) and (e)--(g) above.
For any $\eps>0$ and large enough $x$, $S_x$ has a gap of size at least
$x (\log x)^{C(1)-\eps}$.
\end{theorem}

Clearly, Theorem \ref{thm:main} follows from Theorem \ref{thm: general},
provided that we verify (e), (f) and (g) in the case of polynomial sieving systems.
This verification is accomplished in the next section.
The following sections are devoted to the proof of Theorem \ref{thm: general}.
As noted, the main new idea is to utilize all of the sets $I_q$ for large $q$,
which is encoded in a certain weight function; see \eqref{lambda}
for specifics. 
Hypothesis (f) will be needed at the end of
Section \ref{sec:reduction-to-concentration}
and near the end of Section \ref{sec:correlations}, while hypothesis (g) will
be needed for the proof of the crucial  Lemma \ref{lem:sumE}.

%
\section{Verifying the hypotheses of Theorem \ref{thm: general} for polynomial sieving systems}
%

Item (e) is an immediate corollary of the Chebotarev density theorem.
In fact, $\rho_\nu$ is precisely the proportion of elements of the Galois
group of $f$ which have exactly $\nu$ fixed points.

To verify (f) and (g), we introduce an auxiliary polynomial
$F$ which has roots which are the differences of the roots of $f$.
 By P\'olya's theorem \cite{Polya},
$f$ is integer valued at integers if and only if $f$ 
has the form $f(x)=\sum_{j=0}^d a_j \binom{x}{j}$ with every $a_j\in \Z$.  In particular,  there exists a minimal positive integer $t|d!$ such that $tf\in \Z[x]$. Let also $r_1,...,r_d$ be the complex roots of $f$. Writing
$$
\tilde{f}(x):=tf(x)= c x^d + c_{d-1} x^{d-1} + \cdots + c_1 x + c_0 =
c\prod_{i=1}^d(x-r_i),
$$
where $c=c_d,c_0,\ldots,c_{d-1}\in\Z$, we define the polynomial
\begin{equation}\label{F-defn} 
F(x)=c^{d^2+d}\prod_{1\leq i, j\leq d}\Big(x-(r_i-r_j)\Big)=c^{d^2}\prod_{i=1}^{d}\tilde{f}(x+r_i),
\end{equation}
so that $\deg F=d^2$. We will need the following properties of $F$.
	
\begin{lem}\label{lem:F} 
The polynomial $F$ obeys the following properties:
\begin{itemize}
\item[(i)]  $F\in\Z[x]$;
\item[(ii)] If $f(a)\equiv f(b)\equiv 0 \pmod q$ for some integers $a$ and $b$, then $F(a-b)\equiv 0\pmod q$;  
\item[(iii)] $F(l)\neq0$ for any $l\in\Z\setminus\{0\}$.
\end{itemize}
\end{lem}	

\begin{proof}
Our proofs utilize the Fundamental Theorem of Symmetric Polynomials (FTSP)
\cite[p.20, Theorem (2.4)]{Sym}, which states that any symmetric polynomial 
$P \in \Z[\bu]$, with $\bu=(u_1,\ldots,u_k)$, is equal to a polynomial in $e_1(\bu),\ldots,e_k(\bu)$
with integer coefficients, where $e_j$ is the $j$-th elementary symmetric 
polynomial.
In particular, by the definition of $\tilde{f}$, $c_j = (-1)^{j+d} c e_{d-j}(\br)$ 
for each $0\le j\le d-1$.  Thus $e_j(cr_1,\ldots,cr_d) \in \Z$ for all $1\le j\le d$.

We start with the first claim.  By \eqref{F-defn},
\[
F(x) = \prod_{i=1}^d \bigg[ \sum_{j=0}^d c^{d-j} (cx+cr_i)^j c_j \bigg],
\]
whose coefficients are evidently symmetric polynomials in $(cr_1,\ldots,cr_d)$
with integer coefficients.  By FTSP, $F\in\Z[x]$.

Now we turn to the second claim. Fix $a\in\Z$. We have
$$
\tilde{f}(x)=(x-a)g(x)+\tilde{f}(a)
$$ 
for some polynomial $g\in\Z[x]$ of degree $d-1$ depending on $a$, and therefore by \eqref{F-defn}, 
$$
F(x)=c^{d^2}\prod_{i=1}^d(x+r_i-a)g(x+r_i)+\tilde{f}(a)h(x),
$$
where, by another application of FTSP, $h\in \Z[x]$.
A similar argument shows that 
\[
c^{d(d-1)} \prod_{i=1}^d g(x+r_i)\in \Z[x]
\] 
as well. Thus, for any $b\in\Z$,
$$
F(a-b)=c^{d^2-d}(-1)^{d}\tilde{f}(b)\prod_{i=1}^dg(a-b+r_i)+\tilde{f}(a) h(a-b).
$$
Therefore, $f(a)\equiv f(b)\equiv 0\pmod q$ implies
$\tilde{f}(a) \equiv \tilde{f}(b) \equiv 0 \pmod{q}$ and hence  $q|F(a-b)$ for such $q$, as needed.

For the third claim, let us assume for a contradiction that $F(l)=0$ for some integer $l\neq0$. It means that there is $r_0\in\mathbb{C}$ so that $f(r_0+l)=f(r_0)=0$. But then the polynomial
$$
g(x)=f(x+l)-f(x)
$$
also vanishes at the point $x=r_0$ and $g\notequiv 0$ (otherwise $r_0+kl$ would be zero of $f$ for any $k$, so that $f\equiv0$). Clearly, we also have $\deg g<\deg f$. But this is impossible, since $f$ is irreducible and thus the minimal polynomial of $r_0$ is $\tilde{f}/c$.
\end{proof}

We also need a classic theorem of Erd\H os \cite{Erdos52}
about the average size of the number of divisors of polynomials.
 As is usual, $\tau(n)$ stands for the number of positive divisors of $n$.

\begin{lem}\label{lem:Erdos} 
For any irreducible polynomial $g\in\Z[x]$,
$$
\sum_{k\leq x}\tau(|g(k)|) \ll x\log x.	
$$
\end{lem}

Let $\omega(n)$ denote the number of distinct prime factors of the nonzero integer $n$. If $q$ is prime and $v \bmod q \in I_q-I_q$, Lemma \ref{lem:F} (ii)
implies that $q|F(v)$.  Lemma \ref{lem:F} (iii) implies that $F(v) \ne 0$ if $v\ne 0$. Since  $|F(v)| \ll v^{d^2}$, $\omega(F(v)) \ll \log v + 1$ and this proves  (f).

Now we verify (g).  If $m_q \mod q \in I_q-I_q$ then $q|F(m_q)$
by Lemma \ref{lem:F} (ii), and 
 if $m_q\ne 0$ then $F(m_q)\ne 0$ by Lemma \ref{lem:F} (iii).
 We let $m=m_q+w$.
Thus, if $m$ satisfies $0<|m| \le 2u$, $|w|\le u$ and $m\ne w$, there are $O(\log u)$ 
primes dividing $F(m-w)$; that is, $O(\log u)$ primes $q$ with $m_q+w=m$.
  Also, if $m\bmod p \in I_p-I_p$
implies that $p|F(m)$.  Also $F$ is the product of at most $d^2$
irreducible factors, say $F=F_1\ldots F_s$ with each $F_i$ irreducible.
Hence, if $N(m) \ge k$ then there are at least $k$ distinct primes 
$p$ dividing $F(m)$, and therefore, for some $i$,
at least $k/s$ distinct primes dividing $F_i(m)$.
Hence, using Lemma \ref{lem:Erdos},
\begin{align*}
\# \{ u^{1/2}< q \le u : N(m_q+w)\ge k \} &\ll \# \{ 0<|m| \le 2u : N(m)\ge k\} \log u\\
&\le (\log u) \sum_{i=1}^s \# \{ 0 < |m| \le 2u : \omega(F_i(m)) \ge k/s \} \\
&\le  (\log u) \sum_{i=1}^s 2^{-k/s} \sum_{0<|m| \le 2u} \tau(F_i(m)) \\
&\ll 2^{-k/d^2} u \log^2 u.
\end{align*}
This proves (g), with $c_1=2$ and $c_2=(\log 2)/d^2$,
 and completes the verification of the hypotheses of
Theorem \ref{thm: general} for polynomial sieving systems.

%
%
%
\section{Notation and basic setup}\label{sec2}
%
%
%

We use notation similar to that of \cite{sieved}, with the most important change being the modification of the weight function $\bl$ (see \eqref{APJqn} and \eqref{lambda}
 below). Throughout the proof, we will use positive parameters $K$, $\xi$, $M$ which we describe below; one may think of them as being fixed for most of the time (in fact, it 
 is only the end of Section \ref{sec:reduction-to-concentration} where the exact choice of them is important). The implied constant in $\ll$ and related order estimates may
depend on these parameters.  We will rely on probabilistic methods; boldface symbols
such as $\S$, $\bl$, $\n$, etc. will denote random variables (sets, functions,
numbers, etc.), and the corresponding non-boldface symbols $S$, $\l$, $n$ will
denote deterministic counterparts of these variables.

For a fixed $\d\in (1/10^3,C(1))$, we define
\begin{equation}\label{y-def}
y=\lceil x(\log x)^{\d} \rceil
\end{equation}
and
\begin{equation}\label{z-def}
z=\frac{y\log\log x}{(\log x)^{1/2}}.
\end{equation}

As in \cite{sieved}, our goal is to find a number $b$ so that $S_x+b$ has no elements
in $[1,y]$, which will show that $S_x$ has a gap of size at least $y$.
This is accomplished in three stages:
\begin{enumerate}
\item (Uniform random stage) First, we choose $b$ modulo $P(z)$ uniformly at random; equivalently, for each
prime $p\le z$ we choose $b\bmod p$ randomly with uniform probability, independently for each $p$.
\item (Greedy stage) Secondly, choose $b$ modulo primes in $(z,x/2]$ randomly, but dependent on
 the choice of $b$ modulo $p$ for $p\le z$.   A bit more precisely, for each prime $q\in (z,x/2]$
  with $|I_q|\ge 1$, we will select $b\equiv b_q\pmod{q}$
so that $\{b_q+a+kq : k\in \Z, a\bmod q\in I_q\} \cap [1,y]$
knocks out nearly as many elements of the random set $(S_z+b)\cap [1,y]$ as possible.  
Unlike the argument in \cite{sieved}, we make use of all of the elements of $I_q$
in this stage.  This is the source of our improved theorems.
\item (Clean up stage) Thirdly, we choose $b$ modulo primes $q\in (x/2,x]$ to ensure that
the remaining elements $m\in (S_{x/2}+b) \cap [1,y]$ do not lie in $(S_x+b)\cap [1,y]$ by matching a unique prime $q=q(m)$ with $|I_q|\ge 1$ to each element $m$ and setting $b\equiv m \pmod{q}$.  Here we do use only a single element of $I_q$, whereas
using all of $I_q$ would not improve our theorem at all.
\end{enumerate}

To handle the Greedy stage (2), we divide the primes in $(z,x/2]$ into subsets,
where the primes in each subset are about the same size and with $|I_q|$ is constant.
Primes with rare values of $|I_q|$ will play an insignificant role in our arguments, 
thus we define
\[
\cN = \{ 1\le \nu \le B : \rho_\nu > 0 \},
\]
so that, by the remarks following Hypothesis (g),
\begin{equation}\label{sum-rhonu}
\sum_{\nu \in \cN} \nu \rho_\nu = 1.
\end{equation} 

For example, for the polynomial sieving system with $f(x)=x^2+1$, we have $\cN = \{2\}$ since $\rho_1=0$ (in fact, the only prime with $|I_p|=1$ is $p=2$).

Let $\xi>1$ be a real number (which we will finally choose to be close to $1$), and
define the set of scales
$$
\HH=\left\{ H\in \{1,\xi, \xi^2,...\}: \frac{2y}{x} \leq H\leq \frac{y}{\xi z}\right\}
$$
so that
\begin{equation}\label{H-bounds}
	2(\log x)^{\d}\leq H\leq \frac{y}{z} = \frac{(\log x)^{1/2}}{\log\log x}
	\qquad (H\in \HH).
\end{equation}

For each $H\in\HH$ and $\nu \in \cN$, let 
$$
\Q_{H,\nu}=\bigg\{q\in \Big( \frac{y}{\xi H}, \frac{y}{H} \Big]: |I_q|=\nu \bigg\}.
$$
Hypothesis (e) implies that for each fixed $H,\nu$ we have the asymptotic
\begin{equation}\label{Q-asymp}
|\Q_{H,\nu}|\sim \rho_{\nu}(1-1/\xi)\frac{y}{H\log x} \qquad (x\to\infty).
\end{equation}

Note that if we denote by $\rho$ the density of primes $p$ with $|I_p|\geq1$, then
by (e),
\begin{equation}\label{2.6} 
\rho = \lim_{x\to\infty} \frac{\# \{p\le x: |I_p|\ge 1\}}{x/\log x}=\sum_{\nu\in \cN} \rho_{\nu}.
\end{equation}

Let also 
$$
\Q_H:=\bigg\{q\in \Big( \frac{y}{\xi H}, \frac{y}{H} \Big]: |I_q|\in \cN \bigg\}
=\bigcup_{\nu\in \cN} \Q_{H,\nu}
$$
and, for $\nu\in \cN$,
\[
\Q^{\nu}:=\bigcup_{H\in\HH}Q_{H,\nu},
\]
so that
\[
 \Q:=\bigcup_{H\in\HH}\Q_H=\bigcup_{\nu \in \cN} \Q^{\nu}.
\]
We note that for all $q\in \Q$, $z < q \le x/2$.
Further, for each $q\in \Q$, let $H_q$ be the unique $H$ such that $q\in \Q_H$, which is equivalent to
$$
\frac{y}{\xi H_q}<q\leq \frac{y}{H_q}.
$$
Let also $M$ be a number with
\[
6<M\leq 7,
\]
which we will eventually take to be very close to 6.
We use the notation
$$
S_z=\Z\setminus \bigcup_{p\leq z}I_p,
$$
and
$$
S_{z,x}=\Z\setminus \bigcup_{z<p\leq x}I_p,
$$
and also adopt the abbreviations
\begin{equation}\label{2.8}
P=P(z)=\prod_{p\leq z}p, \quad \s=\s(z):=\prod_{p\leq z}\left(1-\frac{|I_p|}p\right), \quad \S=S_z+\b,  
\end{equation}
where $\b$ is a residue class chosen uniformly at random from $\Z/P(z)\Z$; so, $\S$ is a random shift of $S_z$. For a fixed $H\in \HH$, we also define
\begin{equation}\label{2.9}
	P_1=\prod_{p\leq H^M}p, \quad \s_1=\s(H^M), \quad \b_1\equiv \b \pmod {P_1}, \quad \S_1=S_{H^M}+\b_1,   
\end{equation}
and
\begin{equation}\label{2.10}
	P_2=\prod_{H^M<p\leq z}p, \quad \s_2=\frac{\sigma(z)}{\sigma(H^M)}, \quad \b_2\equiv \b \pmod {P_2}, \quad \S_2=S_{H^M,z}+\b_2, . 
\end{equation}
Obviously, for each $H\in\HH$,
\begin{equation}\label{2.11}
	P=P_1P_2, \quad \s=\s_1\s_2, \quad \S=\S_1\cap \S_2.
\end{equation}
Note that all the quantities defined in (\ref{2.9}) and (\ref{2.10}) depend on $H$ and $M$; however, we will not indicate this dependence for brevity (the values of $H$ and $M$ will always be clear from context). 

Finally, let $\nu_q = |I_q|$ for primes $q$.
 For primes $q\in \Q^{\nu}$, let 
 \[
 I_q=\{a_{1,q} \bmod q,...,a_{\nu,q} \bmod q\} \qquad (a_{i,q}\in[1,q]\cap\Z, \, 1\le i\le \nu).
 \]
  We set 
\begin{equation}\label{APJqn}
	\bAP(J;q,n)=\Bigg(\bigsqcup_{i=1}^{\nu_q}\{n+a_{i,q}+hq: 1\leq h\leq J\}\Bigg)\cap \S_1,
\end{equation}
this being a significant departure from \cite{sieved}.  Here $\bAP(J;q,n)$ is
a portion of $\nu_q$ residue classes modulo $q$.
Also define
\begin{equation}\label{lambda}
\bl(H;q,n)=\frac{\one_{\bAP(KH;q,n)\subset \S_2}}{\s_2^{|\bAP(KH;q,n)|}},
\end{equation}
where $K$ is a positive integer which will be chosen large enough,
and $\one_X$ is the indicator function of a statement $X$. So, for each $q\in\Q$, the weights $\bl(H;q,n)$ are random functions which depend on $\b$. 
Heuristically, $\bl(H;q,n)$ has mean approximately 1, since the probability
that a given set $Y$ lies in $\S_2$ is about $\sigma_2^{|Y|}$; see 
Lemma \ref{lem5.1} below  for a precise statement.

\subsection{General notational conventions.}
The notation $f=O(g)$ and $f\ll g$ mean that $f/g$ is bounded.
The notation $f=O_{\le} (g)$ means that $|f|/g \le 1$.
The notation $o(1)$ stands for a function tending to zero as $x\to \infty$,
at a rate which may depend only on the parameters $\xi, K, M$ and $\delta$,
which we consider to be fixed.  The notation $f \sim g$ means $f=g+o(1)$.
As is usual, $\omega(n)$ is the number of distinct prime factors of $n$.

%
\section{Reduction to concentration of $\l(H;q,n)$} \label{sec:reduction-to-concentration}
%

In this section we deduce Theorem \ref{thm: general} from the following statement.
 Recall that $\S=S_z+\b$ with $\b$ chosen uniformly at random from $\Z/P(z)\Z$.  
 
\begin{theorem}\label{th2} 
Let $\d<C(1)$, $M>6$, $\xi>1$, $K>0$, $0<\eps<\frac17(M-6)$, and assume that $x$ 
is large enough depending on $\d, M,\xi, K,$ and $\eps$. Then there exist a choice of $\b \pmod{P(z)}$ and subsets $\cR^{\nu}\subseteq \Q^{\nu}$ for  $\nu\in \cN$ so that:
\begin{itemize}
\item[(i)]  one has 
\begin{equation}\label{3.1}
|S\cap[1,y]| \leq 2\s y;
\end{equation}
\item[(ii)]
for all $q\in \cR =\bigcup_{\nu\in \cN} \cR^{\nu}$ one has
\begin{equation}\label{3.2}
\sum_{-(K+1)y<n\leq y}\l(H_q;q,n)=\left(1+O_{\leq}\left(\frac{1}{(\log x)^{\d(1+\eps)}}\right)\right)(K+2)y;
\end{equation}
\item[(iii)] for each $\nu \in \cN$ and $i\in\{1,...,\nu\}$, all but at most $\frac{\rho x}{8B^2\log x}$ elements $n$ of $S\cap[1,y]$ obey
\begin{equation}\label{3.3}
\sum_{q\in \cR^{\nu}}\sum_{h\leq KH_q}\l(H_q;q,n-a_{i,q}-qh)= \left(C_{2,\nu}+O_{\leq}\left(\frac{2}{(\log x)^{\d(1+\eps)}}\right)\right)(K+2)y,
\end{equation} 
where $C_{2,\nu}$ is independent of $n$ and $i$ with
\begin{equation}\label{3.4}
C_{2,\nu} \sim \frac{K\rho_{\nu}}{(K+2)M}\frac{1-1/\xi}{\log \xi}\log\left(1/(2\delta)\right) \quad (x\to \infty).
\end{equation}
\end{itemize}
\end{theorem}

\smallskip


We now commence with the deduction of Theorem \ref{thm: general} from Theorem \ref{th2}.
Let $V$ be the set of elements $n$ of $S\cap[1,y]$ for which (\ref{3.3}) holds for all $\nu \in \cN$ and $i=1,...,\nu$. For each $q\in\cR$, we define
the random integer $\n_q$ by the distribution
$$
\P(\n_q=n)=\frac{\l(H_q;q,n)}{\sum_{-(K+1)y<n'\leq y}\l(H_q;q,n')}\qquad
(-(K+1)y < n \le y).
$$
For $q\in \cR$, define the random set
$$
\e_q=\bigsqcup_{i=1}^{\nu_q}\e_{i,q},
$$
where
\be\label{eiq}
\e_{i,q}=V\cap\{\n_q+a_{i,q}+hq: 1\leq h\leq KH_q\}, \quad i=1,...,\nu_q.
\ee
Note that $\e_{i,q}$ and $\e_{j,q}$ are disjoint for $i\neq j$, since all $a_{i,q}$ are distinct modulo $q$. 
To be able to make a clean-up stage, we need to find a choice $n_q$ of $\n_q$ (which corresponds to a choice $e_q$ of the random sets $\e_q$) for each $q\in \cR$, so that for $b$ satisfying $b\equiv n_q \pmod q$ for $q\in\cR$, the estimate
\begin{equation}\label{Sx2-small}
	\Big|(S_{x/2}+b)\cap[1,y]\Big| \leq \frac{\rho x}{4\log x}
\end{equation}
holds. Then by (\ref{2.6}), the number of primes $p\in(x/2,x]$ with $|I_p|\geq1$ is $\sim \frac{\rho x}{2\log x}$ for $x$ large, which guarantees that clean-up stage is possible. To be precise, we may match each element $m \in (S_{x/2}+b)\cap [1,y]$ with a unique prime $p\in (x/2,x]$ with $|I_p|\ge 1$, and choose $b\equiv m-a_{1,p}\pmod{p}$ for each such pair.  Then $(S_x +b) \cap [1,y]=\emptyset$, as desired.

Since $S=S_{x/2}+b$ avoids the $\nu_q=|I_q|$ residues $n_q+a_{1,q},...,n_q+a_{\nu_q,q}$ modulo $q$, we have, by the definition of the set $V$ and \eqref{eiq}, 
\begin{align*}
\Big|(S_{x/2}+b)\cap[1,y]\Big| &\leq \Big|\big((S_{x/2}+b)\cap [1,y]\big)\setminus V\Big|+\Big|V\setminus \bigcup_{q\in \cR}e_q\Big|\\
&\leq \sum_{\nu \in \cN} \sum_{i=1}^{\nu}\frac{\rho x}{8B^2\log x}+\left|V\setminus \bigcup_{q\in \cR}e_q\right| \\
&\leq \frac{\rho x}{4\log x},
\end{align*}
thus verifying \eqref{Sx2-small}, provided that
\begin{equation}\label{3.6}
\left|V\setminus \bigcup_{q\in \cR}e_q\right| \leq 	\frac{\rho x}{8\log x}.
\end{equation}

 To show \eqref{3.6}, we need the following hypergraph covering lemma, which is 
Lemma 3.1 of \cite{sieved}.

\begin{lem}[Hypergraph covering lemma]\label{lem3.1} 
Suppose that $0<\d\leq1/2$ and $K_0\geq1$, and let $y\geq y_0(\d,K_0)$ with $y_0(\d,K_0)$ sufficiently large, and let $V$ be a finite set with $|V|\leq y$. Let $1\leq s\leq y$, and suppose that $\e_1,...,\e_s$ are random subsets of $V$ satisfying the following:
\begin{equation}\label{3.7}
|\e_i|\leq \frac{K_0(\log y)^{1/2}}{\log\log y} \quad (1\leq i\leq s),
\end{equation}
\begin{equation}\label{3.8}
\P(v\in \e_i) \leq y^{-1/2-1/100} \quad (v\in V, 1\leq i\leq s),	
\end{equation}
\begin{equation}\label{3.9}
\sum_{i=1}^s\P(v,v'\in \e_i) \leq y^{-1/2} \quad (v,v'\in V, v\neq v'),	
\end{equation}
\begin{equation}\label{3.10}
\left|\sum_{i=1}^s\P(v\in \e_i)-C_2\right| \leq \eta \quad (v\in V),	
\end{equation}	
where $C_2$ and $\eta$ satisfy
\begin{equation}\label{3.11}
10^{2\d} \leq C_2 \leq 100, \quad \eta \geq \frac{1}{(\log y)^{\d}\log\log y} .	
\end{equation}
Then there are subsets $e_i$ of $V$, $1\leq i\leq s$, with $e_i$ being in the support of $\e_i$ for every $i$, and such that
\begin{equation}\label{3.12}
\left|V\setminus \bigcup_{i=1}^se_i\right| \leq C_3\eta|V|,
\end{equation}
where $C_3$ is an absolute constant.
\end{lem}

We apply this lemma with $s=|\cR|$, $\{\e_i: i=1,...,s\}=\{\e_q: q\in\cR\}$, $K_0=BK$, and
$$
\eta=\frac{\rho}{20 C_1 C_3(\log x)^{\d}}.
$$
By \eqref{Mertens} and \eqref{z-def}, we have $\sigma=\sigma(z) \sim C_1/\log z \sim C_1/\log x$.
The conclusion of the lemma (with the bound (\ref{3.1})) implies that there is a choice of sets $e_q$ with  
$$
\left|V\setminus \bigcup_{q\in \cR}e_q\right| \leq C_3\eta|V|
=\frac{\rho|V|}{20 C_1 (\log x)^{\d}} \leq \frac{\rho \s y}{10 C_1 (\log x)^{\d}} \leq 
\frac{\rho x}{8\log x},
$$
which is enough for (\ref{3.6}), so that we are left to verify the conditions of the lemma. First of all, by \eqref{H-bounds}, we have
$$
|\e_q| = \sum_{i=1}^{\nu_q}|\e_{i,q}|\leq \nu_q KH_q \leq \frac{BKy}{z} \leq \frac{BK(\log x)^{1/2}}{\log\log x} \leq \frac{BK(\log y)^{1/2}}{\log\log y},
$$
which gives us (\ref{3.7}). 
For each $n\in V$ and $q\in \cR$, (\ref{3.2}) together with the trivial bound  $\l(H;q,n) \le \sigma_2^{-BH_q} \le y^{o(1)}$ (from \eqref{H-bounds}, \eqref{APJqn} and \eqref{lambda})  gives us
\be\label{Pneq}
\begin{split}
\P(n\in \e_q)&=\sum_{1\leq h\leq KH_q} \sum_{i=1}^{\nu_q} \P(\n_q=n-a_{i,q}-hq)\\ 
&\ll \frac1y\sum_{1\leq h\leq KH_q}\sum_{i=1}^{\nu_q}\l(n-a_{i,q}-hq)  \\
&\ll y^{-0.999},
\end{split}
\ee
which verifies (\ref{3.8}). Now we turn to (\ref{3.10}). From (\ref{3.2}), (\ref{3.3}), and (\ref{3.4}), followed by an application of \eqref{sum-rhonu}, we obtain
\begin{multline*}
\sum_{q\in \cR}\P(n\in \e_q)=\sum_{\nu\in \cN}\sum_{q\in\cR^{\nu}}\P(n\in \e_q)=\sum_{\nu\in \cN}\nu C_{2,\nu}+O((\log x)^{-\d(1+\eps)})=\\
=\sum_{\nu\in \cN}\nu \rho_{\nu}\frac{K(1-1/\xi)}{(K+2)M\log\xi}\log(1/(2\d))+o(1)= C_2+o(1),
\end{multline*}
where we define
$$
C_2=\frac{K(1-1/\xi)}{(K+2)M\log\xi}\log(1/(2\d)).
$$
Recalling that $\d<C(1)$ together with the definition \eqref{Crho-def} of $C(1)$, we see
 that $C_2$ is at least $10^{2\d}$ provided that $M-6$ and $\xi-1$ are sufficiently small in
  terms of $\d$, $K$ is sufficiently large in terms of $\d$, 
  $0<\eps<\frac17(M-6)$, and $x$
   is large enough depending on $\d,M,\xi,K,\eps$. Also $C_2\leq100$ due to $\d\geq10^{-3}$. Thus, (\ref{3.10}) follows.

It remains to check that (\ref{3.9}) holds. We take any distinct $v,v'\in V$ and see that 
$$
\sum_{q\in\cR}\P(v,v'\in \e_q) \leq \sum_{q\in \cR}\Bigg(\sum_{i=1}^{\nu_q}\P(v,v'\in \e_{i,q}) +\ssum{1\le i,j \le \nu_q \\ i\ne j} \P(v\in\e_{i,q}, v'\in \e_{j,q})\Bigg).	
$$
If both $v,v'$ both belong to some $\e_{i,q}$, then $q$ divides $v-v'$.
Since $V\subseteq [1,y]$, $0<|v-v'|\leq y$ and also $q>z>y^{3/4}$, hence there is at most one such $q$. 
Further, if $v\in \e_{i,q}$ and $v'\in \e_{j,q}$ for some $q$ and $i\neq j$, then $v-v'\equiv a_{i,q}-a_{j,q} \pmod q$ and hence $v-v' \bmod q \in I_q-I_q$.
By hypothesis (f) and the bound $|v-v'|\le y$, the number of such $q$ is $\ll y^{0.49}$.
Thus, by \eqref{Pneq},
\[
\sum_{q\in\cR}\P(v,v'\in \e_q) \ll y^{0.49} \cdot \max_{v,i,q} \P(v\in \e_{i,q})
\ll y^{-0.509},
\]
which gives (\ref{3.9}). 

\medskip 

Thus, we verified the conditions of Lemma \ref{lem3.1}, and (\ref{3.6}) follows.
This completes the proof of Theorem \ref{thm: general} assuming Theorem \ref{th2}.

%
\section{Concentration of $\l(H;q,n)$} \label{sec:concentration}
%

In this section we reduce Theorem \ref{th2} to the following assertion.

\begin{theorem}\label{th3}
Let $M>2$, $K>0$, and $\xi>1$. Then
	
$(i)$ One has 
	
\begin{equation}\label{thm4(i)}
\E\Bigl|\S\cap[1,y]\Bigr|=\s y;  \quad 		\E\Bigl|\S\cap[1,y]\Bigr|^2=\left(1+O\left(\frac{1}{\log y}\right)\right)(\s y)^2;
\end{equation}
	
\smallskip	
	
$(ii)$ For every $H\in \HH$, every $\nu\in \cN$ and $j\in\{0,1,2\}$,	
	
\begin{equation}\label{thm4(ii)}
\E\sum_{q\in \Q_{H,\nu}}\Bigg(\sum_{-(K+1)y<n\leq y}\bl(H;q,n)\Bigg)^j = \left(1+O\left(\frac{\log H}{H^{M-2}}\right)\right)\big((K+2)y\big)^j|\Q_{H,\nu}|;
\end{equation}	
	
\smallskip		
	
$(iii)$ For every $H\in \HH$, every $\nu\in\cN$, $i\in\{1,...,\nu\}$, and $j\in\{0,1,2\}$, 
	
\begin{multline}\label{thm4(iii)}
\E\sum_{n\in \S\cap[1,y]}\Bigg(\sum_{q\in \Q_{H,\nu}}\sum_{h\leq KH}\bl(H;q,n-a_{i,q}-qh)\Bigg)^j = \\ \left(1+O\left(\frac{\log H}{H^{M-2}}\right)\right)\left(\frac{|\Q_{H,\nu}|\cdot\lfloor KH\rfloor }{\s_2}\right)^j\s y.	
\end{multline}		
	
\end{theorem}

\smallskip 

\begin{proof}[Deduction of Theorem \ref{th2} from Theorem \ref{th3}]
	
Fix $\d<C(1)$, $M>6$, $\xi>1$, $K>0$, and also $0<\eps<\frac17(M-6)$. From Theorem \ref{th3} (i) we have
$$
\E \Big||\S \cap [1,y]| - \s y\Big|^2 \ll \frac{(\s y)^2}{\log y}.
$$
Hence by Chebyshev’s inequality, we see that
\begin{equation}\label{4.4}
\P (|\S \cap [1, y]| \leq 2\s y) = 1 - O(1/\log x),
\end{equation}
showing that (\ref{3.1}) in Theorem \ref{th2} holds with probability $1-o(1)$. 	
	
Now we work on parts (ii) and (iii) of Theorem \ref{th2}. For each $H\in\HH$ and $\nu\in \cN$, we have from (\ref{thm4(ii)}) 
\begin{equation}\label{4.5}
\E\sum_{q\in \Q_{H,\nu}}\left(\sum_{-(K+1)y<n\leq y}\bl(H;q,n) - (K+2)y\right)^2 \ll \frac{y^2|\Q_{H,\nu}|}{H^{M-2-\eps}}.
\end{equation}	
Now let $\bcR_{H,\nu}$ be the (random) set of $q\in \Q_{H,\nu}$ for which
\begin{equation}\label{4.6}
\left|\sum_{-(K+1)y<n\leq y}\bl(H;q,n) - (K+2)y\right| \leq \frac{y}{H^{1+\eps}}.
\end{equation}	
By estimating the left-hand side of (\ref{4.5}) from below by the sum over $q\in \Q_{H,\nu}\setminus\bcR_{H,\nu}$, we find that
\begin{equation}\label{4.7}
\E|\Q_{H,\nu}\setminus\bcR_{H,\nu}|	\ll \frac{|\Q_{H,\nu}|}{H^{M-4-3\eps}}.
\end{equation} 
We let
$$
\bcR^\nu=\bigcup_{H\in\HH}\bcR_{H,\nu},
$$
and then Theorem \ref{th2} (ii) follows from the lower bound on $H$ given in \eqref{H-bounds}.

We now turn to the condition (iii) of Theorem \ref{th2}. Similarly to (\ref{4.6}), for each $H\in\HH$, $\nu\in \cN$, and $i\in\{1,...,\nu\}$, from (\ref{thm4(iii)}) we have
\begin{multline}\label{4.8}
\E\sum_{n\in \S\cap[1,y]}\left(\sum_{q\in \Q_{H,\nu}}\sum_{h\leq KH}\bl(H;q,n-a_{i,q}-qh)-\frac{|\Q_{H,\nu}|\cdot\lfloor KH\rfloor}{\s_2}\right)^2 \ll \\
\frac{1}{H^{M-2-\eps}}\left(\frac{|\Q_{H,\nu}|\cdot\lfloor KH\rfloor}{\s_2}\right)^2\s y.	
\end{multline}		
Let $\bmE_{H,\nu,i}$ be the set of $n\in \S\cap[1,y]$ such that
\begin{equation}\label{4.9}
\left|\sum_{q\in \Q_{H,\nu}}\sum_{h\leq KH}\bl(H;q,n-a_{i,q}-qh)-\frac{|\Q_{H,\nu}|\cdot\lfloor KH\rfloor}{\s_2}\right| \geq \frac{|\Q_{H,\nu}|\cdot\lfloor KH\rfloor}{\s_2H^{1+\eps}}.	
\end{equation}		
Then, since $M>6$ and $\eps<(M-6)/7$, (\ref{4.8}) implies that
$$
\E|\bmE_{H,\nu,i}| \ll \frac{\s y}{H^{M-4-3\eps}} \ll \frac{\sigma y}{H^2},
$$
and, hence, $|\bmE_{H,\nu,i}|\leq \s y/H^{1+\eps}$ with probability $1-O(H^{-1+\eps})$.
	
Now we estimate the contribution from ``bad'' primes $q\in \Q_{H,\nu}\setminus \bcR_{H,\nu}$. For any $h\leq KH$, we get from Cauchy-Schwarz inequality for vector functions
\begin{multline*}
\E\sum_{q\in \Q_{H,\nu}\setminus\bcR_{H,\nu}}\sum_{n\in \S\cap[1,y]}\bl(H;q,n-a_{i,q}-qh) \leq \\
\left(\E|\Q_{H,\nu}\setminus\bcR_{H,\nu}|\right)^{1/2}\left(\E\sum_{\Q_{H,\nu}\setminus\bcR_{H,\nu}}\left|\sum_{-(K+1)y<n\leq y}\bl(H;q,n)\right|^2\right)^{1/2}, 	
\end{multline*}
where we extended the range of summation of $\bl(H;q,\cdot)$ to the larger interval $(-(K+1)y,y]$ (note that $a_{i,q}+qh\leq q+Ky\leq (K+1)y$ and the weights $\bl(H;q,\cdot)$ are non-negative). Further, by the triangle inequality, (\ref{4.5}) and (\ref{4.7}), 
\begin{multline*}
\E\sum_{\Q_{H,\nu}\setminus\bcR_{H,\nu}}\Bigg|\sum_{-(K+1)y<n\leq y}\bl(H;q,n)\Bigg|^2 \leq \\ 
2\E\sum_{\Q_{H,\nu}\setminus\bcR_{H,\nu}}\Bigg(\Bigg|\sum_{-(K+1)y<n\leq y}\bl(H;q,n)-(K+2)y\Bigg|^2+(K+2)^2y^2\Bigg) \ll \frac{y^2|\Q_{H,\nu}|}{H^{M-4-3\eps}}.	
\end{multline*}
Combining two latter estimates, using \eqref{4.7} again, and summing over all $h\leq KH$,
we get 
$$
\E\sum_{n\in \S\cap[1,y]}\sum_{q\in \Q_{H,\nu}\setminus \bcR_{H,\nu}}\sum_{h\leq KH} \bl(H;q,n-a_{i,q}-qh) \ll \frac{y|\Q_{H,\nu}|}{H^{M-5-3\eps}}.
$$
Let $\bmF_{H,\nu,i}$ be the set of $n\in\S\cap[1,y]$ such that
\begin{equation}\label{4.10} 
\sum_{q\in \Q_{H,\nu}\setminus \bcR_{H,\nu}}\sum_{h\leq KH} \bl(H;q,n-a_{i,q}-qh) \geq \frac{|\Q_{H,\nu}|\cdot\lfloor KH\rfloor}{\s_2H^{1+\eps}}.	
\end{equation} 
Then 
$$
\E|\bmF_{H,\nu,i}| \ll \frac{\s_2y}{H^{M-5-4\eps}} \ll \frac{\s y \log H}{H^{M-5-4\eps}},
$$ 
and, by Markov's inequality,
$$
|\bmF_{H,\nu,i}| \leq \frac{\s y}{H^{1+\eps}} 
$$
with probability $1-O(H^{-(M-6-6\eps)})$. Since $\eps<(M-6)/7$, we have $M-6-6\eps>\eps$, and the last probability becomes $1-O(H^{-\eps})$.
	
Since $\sum_{H\in\HH}H^{-\eps}\ll (\log x)^{-\d\eps}$,
with probability $1-o(1)$ we have that
for all $H\in\HH$, $\nu\in\cN$, and $i\in\{1,...,\nu\}$
that both sets $\bmE_{H,\nu,i}, \bmF_{H,\nu,i}$ have size  at most $(\s y)H^{-1-\eps}$.
	
Now we make a choice of $b \pmod{P(z)}$. We consider the event that 
$|\S \cap [1,y]|\le 2\sigma y$ and that for each $H,\nu,i$,  the sets $\bmE_{H,\nu,i}, \bmF_{H,\nu,i}$ have size at most $(\s y)H^{-1-\eps}$. By the above discussion, this event holds with probability at least $1-o(1)$ as $x\to\infty$ (so this probability is at least, say, $1/2$ whenever $x$ is large enough depending on $\d, M,\xi, K,$ and $\eps$). From now, we fix a $b\bmod{P(z)}$ such that it is so, and thus all of our random sets and weights become deterministic.  
With this choice of $b$ we verify condition (iii) in Theorem \ref{th2}.
	
For fixed $\nu\in\cN$ and $i\in\{1,...,\nu\}$, we set
$$
\cM_{\nu,i}=\Big(S\cap[1,y]\Big)\setminus\bigcup_{H\in\HH}\left(\mE_{H,\nu,i}\cup\mF_{H,\nu,i}\right).
$$
Now we verify (\ref{3.3}) with given $\nu$ and $i$ for $n\in\mathcal{M}_{\nu,i}$. 
By \eqref{H-bounds}, $\sum_{H\in \HH} H^{-1-\eps} \ll (\log x)^{-(1+ \eps)\delta}$, and so
 the number of exceptional elements satisfies
$$
\left|\bigcup_{H\in\HH}\left(\mE_{H,\nu,i}\cup\mF_{H,\nu,i}\right)\right| \ll \frac{\s y}{(\log x)^{(1+\eps)\d}},
$$
which, by \eqref{z-def}, is smaller than $\frac{\rho x}{8 B^2\log x}$ for large $x$.
 We fix arbitrary $n\in \mathcal{M}_{\nu,i}$. For such $n$, the inequalities (\ref{4.9}) and (\ref{4.10}) both fail, and therefore for each $H\in\HH$,
$$
\sum_{q\in \cR_{H,\nu}}\sum_{h\leq KH}\l(H;q,n-a_{i,q}-qh)=\left(1+O_{\leq}\left(\frac{2}{(\log x)^{(1+\eps)\d}}\right)\right) \frac{|\Q_{H,\nu}|\cdot\lfloor KH\rfloor}{\s_2}. 
$$
Summing over all $H\in\HH$, we have
\begin{equation*}
\sum_{q\in \cR^{\nu}}\sum_{h\leq KH_q}\l(H_q;q,n-a_{i,q}-qh)=\left(1+O_{\leq }\left(\frac{2}{(\log x)^{(1+\eps)\d}}\right)\right)C_{2,\nu}(K+2)y
\end{equation*}
with (recall that $\s_2$ depends on $H$)
$$
C_{2,\nu}=\frac{1}{(K+2)y}\sum_{H\in\HH} \frac{|\Q_{H,\nu}|\cdot\lfloor KH\rfloor}{\s_2}.
$$
Note that $C_{2,\nu}$ depends on $x,K,M,\xi,$ and $\d$, but not on $n$. Since 
$$
\lfloor KH \rfloor=KH\big(1+O(1/H)\big)=KH\big(1+O(\log x)^{-\d}\big)
$$
and
$$
\s_2^{-1}=\prod_{H^M<p\leq z}(1-|I_p|/p)^{-1}\sim\frac{\log z}{M\log H},	
$$
we get, using \eqref{Q-asymp}, 
$$
C_{2,\nu}\sim \frac{K}{(K+2)y} \cdot \rho_{\nu}(1-1/\xi)\sum_{H\in\HH}\frac{y/H}{\log x}\cdot\frac{H\log z}{M\log H} \sim \frac{\rho_{\nu}K(1-1/\xi)}{M(K+2)}\sum_{H\in\HH}\frac{1}{\log H},
$$
as $x\to\infty$. Recalling the definition of $\HH$, we see that
$$
C_{2,\nu}\sim \frac{\rho_{\nu}K(1-1/\xi)}{M(K+2)\log \xi}\sum_{j}\frac{1}{j},
$$
where $j$ runs over the interval
$$
\frac{\d\log\log x}{\log \xi} \leq j \leq \frac{(1/2+o(1))\log\log x}{\log \xi}.
$$
We thus obtain
$$	
C_{2,\nu}\sim \frac{\rho_{\nu}K(1-1/\xi)}{M(K+2)\log \xi}\log(1/(2\d)), \quad x\to\infty,
$$
and the claim \eqref{3.3} follows.  
\end{proof}

It remains to establish Theorem \ref{th3}. This is the aim of the last section of the paper.

%
%
\section{Computing correlations}  \label{sec:correlations}
%
%

In this section we prove Theorem \ref{th3}. The claim (i) is exactly \eqref{thm4(i)} and \eqref{thm4(ii)} of Theorem 3 from \cite{sieved}. To verify the claims (ii) and (iii),
we must rework the argument from \cite{sieved} using our new weight function
$\bl$ from \eqref{lambda}.
 For $H\in\HH$, let $\D_H$ be the collection of square-free numbers $D$, all of whose prime divisors lie in $(H^M,z]$. 
For each $D \in \D_H$, let $I_D \subset \Z/D\Z$ be defined as $I_D = \bigcap_{p|D} I_p$.
Further, for $A>0$, let 
\begin{equation}\label{5.1}
E_A(m;H)=\big( \one_{m\ne 0} \big)\sum_{D\in \D_H\setminus\{1\}}\frac{A^{\o(D)}}{D}\one_{m \bmod D \,\in\, I_D-I_D}.
\end{equation}
 Note that $E_A(m;H)=E_A(-m;H)$ for all $m\in \Z$.
 Also, this notation differs slightly from that in \cite{sieved}, in that we include
 the factor $\one_{m\ne 0}$ here.  Our notation then makes in unnecessary to explicitly exclude the case $m=0$ from summations.

\medskip 

We need the following lemmas, which are Lemma 5.1 and Lemma 5.2 of \cite{sieved},
respectively.

\begin{lem}\label{lem5.1} 
	Let $10<H<z^{1/M}$, $1\leq l\leq 10KH$, and $\U\subset \V$ be two finite sets of
	integers with $|\V|=l$. Then 	
	$$
	\P(\U\subset\S_2)=\s_2^{|\U|}\Biggl(1+O\Bigl(|\U|^2H^{-M}+l^{-2}\sum_{v,v'\in\V}E_{2l^2B}(v-v';H)\Bigr)\Biggr).
	$$		
\end{lem}

\begin{lem}\label{lem5.2} 
Let $10<H<z^{1/M}$, $0<AB^2 \le H^M$ and $(m_t)_{t\in T}$ be a finite sequence such that
\begin{equation}\label{5.2}
\sum_{t\in T}\one_{m_t\equiv a\pmod{D}} \ll \frac{X}{\varphi(D)}+R	
\end{equation}
for some $X,R>0$, and all $D\in \D_H\setminus\{1\}$ and $a\in\Z/D\Z$. Then for any integer $j$	
$$
\sum_{t\in T}E_A(m_t+j;H) \ll \frac{XA}{H^M}+R\exp(AB^2\log\log y).
$$
\end{lem}

For the rest of the paper we use the notation  
$$
A=A(H)=8B^3K^2H^2
$$ 
for the brevity, and also for prime $q$, define the set
$$
\A_q=\Big\{a_{i,q}-a_{j,q}\,\big|\, 1\leq i\leq j\leq \nu \Big\}.
$$ 
Recalling that $a_{i,q}\in[1,q]$, we see that $\A_q\subset [1-q,q-1]$.

 We will need the following bound, which is where we deploy Hypothesis (g).

\begin{lem}\label{lem:sumE}
Let $\nu\in\cN$ and $H\in \HH$. 
For $q\in \Q_{H,\nu}$, suppose that $m_q \bmod q \in I_q-I_q$
with $0<|m_q| \le x\log x$, and suppose that
$w\in \Z$ with $|w| \le x\log x$.  Then	
$$
\sum_{q\in\Q_{H,\nu}}E_{A(H)}(m_q+w;H)	\ll \frac{|\Q_{H,\nu}| \log H}{H^{M-2}}.
$$
\end{lem}	

\begin{proof} 
If $m_q \bmod D \in I_D-I_D$, then $m_q \bmod p \in I_p-I_p$ for each $p|D$.
Thus, if $m_q+w\ne 0$ then
\begin{equation}\label{5.3}
\begin{split}
E_A(m_q+w;H)&=\prod_{\substack{m_q +w \bmod p \,\in\, I_p-I_p\\H^M<p\leq z}}\left(1+\frac{A}{p}\right)-1 \\
&\leq \exp\Bigg(A\sum_{\substack{m_q+w \bmod p \,\in\, I_p-I_p \\ H^M<p\leq z}\;}\frac1p\Bigg)-1.
\end{split}
\end{equation}
Recall the notation $N(m)$ from Hypothesis (g).
Thus, the number of primes $p$ with $H^M < p \le z$ and with $m_q+w \bmod p \,\in\, I_p-I_p$ is at most $N(m_q+w)$.
Let $c_3$ be a sufficiently large constant, depending on $c_1$ and $c_2$ from Hypothesis (g), and let
$$
\widetilde{\Q}_{H,\nu}=\big\{q\in\Q_{H,\nu}: N(m_q+w) \leq c_3 \log H \big\}.
$$
Clearly, for $q\in\widetilde{\Q}_{H,\nu}$ we have 
$$
\sum_{\substack{m_q+w \bmod p \,\in\, I_p-I_p \\ H^M<p\leq z}} \frac1p \leq \frac{c_3\log H}{H^M}.
$$
Therefore, using the fact that $A=O(H^2)$,
\[
\ssum{q\in\widetilde{\Q}_{H,\nu}} E_A(m_q+w;H) \ll |\widetilde{\Q}_{H,\nu}|
\Big(\exp(O((\log H) H^{-(M-2)}))-1\Big)  \ll \frac{|\Q_{H,\nu}|\log H}{H^{M-2}}. 
\]
Using Hypothesis (g), for some positive constants $c_1,c_2$ (depending only
on the sieving system),
\begin{align*}
\ssum{q\in \Q_{H,\nu}\setminus \widetilde{\Q}_{H,\nu}}E_A(m_q+w;H) &\le 
\sum_{k>c_3\log H} \#\big\{q\in \Q_{H,\nu}: m_q+w\ne 0,\, N(m_q+w)=k\big\} e^{Ak/H^{M}} \\
&\ll x (\log x)^{c_1+1} \sum_{k>c_3 \log H} e^{-c_2 k}\exp\Big(O(kH^{-(M-2)})\Big) \\
&\ll x (\log x)^{c_1+1} e^{-c_2 c_3 \log H}  \\
&\ll |\Q_{H,\nu}| H^{-(M-2)},
\end{align*}
if $c_3$ is large enough, using the lower bound  $H\geq (\log x)^{\d}$ 
 from \eqref{H-bounds} and the asymptotic \eqref{Q-asymp}; recall also that $6<M\leq 7$.  This concludes the proof. 
\end{proof}

Now we fix $H\in\HH$ and $\nu\in\cN$ for the rest of the paper. We start with the proof of part (ii) in the case $j=1$ (the case $j=0$ being trivial), which is 
\begin{equation}\label{5.5}
\E\sum_{q\in \Q_{H,\nu}}\sum_{-(K+1)y<n\leq y}\bl(H;q,n) = \left(1+O\left(\frac{\log H}{H^{M-2}}\right)\right)(K+2)y|\Q_{H,\nu}|.
\end{equation}	
By \eqref{lambda}, the left-hand side expands as
\begin{equation*}
\E\sum_{q\in\Q_{H,\nu}}\;\sum_{-(K+1)y<n\leq y}\frac{\one_{\bAP(KH;q,n)\subset\S_2}}{\s_2^{|\bAP(KH;q,n)|}}.
\end{equation*}
Recall that, according to the definitions (\ref{2.9}) and (\ref{2.10}), $\b_1$ and $\b_2$ are independent, and so are $\bAP(KH; q,n)$ and $\S_2$.  With $\b_1$ fixed,
 $\bAP(KH;q,n)$ is also fixed and we will denote it as $AP(KH;q,n)$.
 Then the above expression equals
$$  
\sum_{q\in\Q_{H,\nu}}\;\sum_{-(K+1)y<n\leq y}\;\sum_{b_1\bmod {P_1}}\;\frac{\P(\b_1=b_1)}{\s_2^{|\AP(KH;q,n)|}}\P(\AP(KH; q,n)\subset\S_2).
$$
For fixed $q$, $n$, and $b_1$, we apply Lemma \ref{lem5.1} to the sets $\U=\AP(KH;q,n)$ and 
$$
\V=\bigsqcup_{i=1}^{\nu}\{n+a_{i,q}+qh: 1\leq h\leq KH\},
$$ 
so that $l=|\V|=\nu \fl{KH}\asymp H$.
Since $E_A(m;H)$ is an increasing function of $A$,
we find that the left-hand side of (\ref{5.5}) is equal to
\begin{multline}\label{ii-j1-errors}
\sum_{q\in \Q_{H,\nu}}\sum_{-(K+1)y<n\leq y}\Big[1+O\big(H^{-(M-2)}\big)\Big]+\\
+O\bigg( yH^{-2} \sum_{q\in \Q_{H,\nu}}\; \sum_{a\in \A_q}\;
\ssum{1\leq h\le h'\leq KH} E_{A(H)}(a+qh-qh';H)\bigg).
\end{multline}
We note that $|a+q(h-h')|\le (K+1)y \le x\log x$ for $a\in\A_q$ and large $x$.
When $h$ and $h'$ are fixed, we apply Lemma \ref{lem:sumE} with $w=0$
and $m_q=a+qh-qh'$, where we've chosen one of the $O(1)$ elements $a\in \cA_q$
for each $q$.
Thus, we see that the second line in \eqref{ii-j1-errors} is
\[
\ll (y H^{-2}) H^2 |\cQ_{H,\nu}| \cdot (\log H) H^{-M+2} = \frac{y |\cQ_{H,\nu}| \log H}{H^{M-2}}.
\]
This proves the $j=1$ case of part (ii) in Theorem \ref{th3}, that is, \eqref{5.5}.
\bigskip

Now we turn to the case $j=2$ of (ii), which is 
$$
\E\sum_{q\in \Q_{H,\nu}}\bigg(\sum_{-(K+1)y<n\leq y}\bl(H; q,n)\bigg)^2=
\left(1+O\left(\frac{\log H}{H^{M-2}}\right)\right)(K+2)^2y^2|\Q_{H,\nu}|. 
$$ 
The left-hand side is expanded as
$$
\E\sum_{q\in\Q_{H,\nu}} \;\sum_{-(K+1)y<n_1,n_2\leq y}\frac{\one_{\bAP(KH;q,n_1)\cup \bAP(KH;q,n_2)\subset \S_2}}{\s_2^{|\bAP(KH;q,n_1)|+|\bAP(KH;q,n_2)|}}. 
$$ 
For fixed $q,n_1,n_2$, we will apply Lemma \ref{lem5.1} with
$$
U=\bAP(KH;q,n_1)\cup \bAP(KH;q,n_2)
$$
and
$$
V=V_1\cup V_2,
$$
where
$$
V_j=\bigsqcup_{i=1}^{\nu}\{n_j+a_{i,q}+qh: 1\leq h\leq KH\}, \quad j=1,2. 
$$
We first estimate the contribution of the triples $(n_1,n_2,q)$ for which $V_1$ and $V_2$ have non-empty intersection. This implies that $(n_1-n_2) \bmod q \in \A_q$, and, hence, there are $O(yH)$ such pairs $n_1,n_2$ for each $q$. Each of them contributes at most $\s_2^{-4KH}=y^{o(1)}$, so the total contribution of such triples is 
$O(y^{1+o(1)} |\Q_{H,\nu}|)$, which is negligible. Thus we may restrict our attention 
to those triples $(n_1,n_2,q)$ for which the sets $V_1$ and $V_2$ do not intersect; let us call these triples \textit{good}.
 In particular, for any good triple $(n_1,n_2,q)$, the sets $\bAP(KH;q,n_1)$ and 
 $\bAP(KH;q,n_2)$ also do not intersect. Then it is enough to show that
\begin{multline}\label{5.7}
\E\sum_{q\in\Q_{H,\nu}}\;\sum_{\substack{-(K+1)y<n_1,n_2\leq y\\(n_1,n_2,q) \, \text {good}}}\frac{\one_{\bAP(KH;q,n_1)\sqcup \bAP(KH;q,n_2)\subset\S_2}}{\s_2^{|\bAP(KH;q,n_1)|+|\bAP(KH;q,n_2)|}}\\
=\left(1+O\left(\frac{\log H}{H^{M-2}}\right)\right)(K+2)^2y^2|\Q_{H,\nu}|.
\end{multline}
Arguing as in the case $j=1$, we see that the left-hand side of (\ref{5.7}) equals 
\begin{equation}\label{5.8} 
\begin{split}
&\sum_{q\in\Q_{H,\nu}}\sum_{\substack{-(K+1)y<n_1,n_2\leq y\\(n_1,n_2,q)\, \text {good}}}\Bigl(1+O\Bigl(\frac{1}{H^{M-2}}\Big) \Big)+ \\
&\qquad \qquad + O \Bigg(\frac{1}{H^2}
\sum_{q\in\Q_{H,\nu}}\;\sum_{-(K+1)y<n_1,n_2\leq y}R_0(n_1,n_2,q)\Bigg),
\end{split}
\end{equation}
where
\begin{align*}
	R_0(n_1,n_2,q) &:= \ssum{v,v' \in V \\ v\ne v'} E_{A(H)}(v-v';H) \\
	&\ll \sum_{1\le h,h' \le KH} \; \sum_{a\in \cA_q} E_{A(H)}(n_1-n_2+a+q(h-h');H)\\
	&\ll \ssum{1\le h,h' \le KH \\ a\in \cA_q \\ a\ne 0\text{ or } h\ne h'} E_{A(H)}(n_1-n_2+a+q(h-h');H) + H\, E_{A(H)} (n_1-n_2) \\
	&= R_1(n_1,n_2,q) + H\, E_{A(H)} (n_1-n_2),
\end{align*}
say.  
Recalling that all but $O(yH)$ pairs $(n_1,n_2)$ are good, we get the main term $(K+2)^2y^2|\Q_{H,\nu}|$ from the first line of (\ref{5.8}), with an acceptable error
term.

We next estimate the contribution from $R_1(n_1,n_2,q)$.
With $n_1,n_2,h_1,h_2$ fixed and also fixing one of the $O(1)$ choices for $a\in A_q$ for
each $q\in \cQ_{H,\nu}$, we estimate $\sum_{q\in \cQ_{H,\nu}} R_1(n_1,n_2,q)$
using Lemma \ref{lem:sumE} with $w=n_1-n_2$ and $m_q=a+q(h-h')$. 
Since either $a\ne 0$ or $h\ne h'$, we have $m_q \ne 0$.
Also, for large $x$, $|w|\le x\log x$ and $|m_q| \le x\log x$.
  Therefore,
\begin{align*}
\sum_{n_1,n_2} \; \sum_{q\in \cQ_{H,\nu}} R_1(n_1,n_2,q) &\ll H^2 y^2\, \frac{|\cQ_{H,\nu}|\log H}{H^{M-2}},
\end{align*}
which is acceptable for \eqref{5.7}.
To estimate the contribution from $E_{A(H)}(n_1-n_2;H)$,
we apply Lemma \ref{lem5.2}, by first fixing $n_2$ and 
observing that \eqref{5.2} holds with $X=y$ and $R=1$.
Therefore, recalling that $A(H) \ll H^2$,
\begin{align*}
\sum_{q\in \cQ_{H,\nu}} \; \sum_{-(K+1)y < n_1,n_2 \le y}\; E_{A(H)}(n_1-n_2;H) &\ll
|\cQ_{H,\nu}|\, y \bigg( \frac{y}{H^{M-2}} + e^{AB^2\log\log y} \bigg) \\
&\ll \frac{y^2|\cQ_{H,\nu}|}{H^{M-2}},
\end{align*}
which is also acceptable for \eqref{5.7}.
This gives \eqref{5.7}, as desired, completing the $j=2$ case of (ii).

\smallskip

\begin{proof}[Proof of (iii)]  Fix $H\in \HH$, $\nu \in \cN$ and $1\le i\le \nu$.
 The case $j=0$ follows from part (i), so we focus on the case $j=1$, which states
$$
\E\sum_{n\in \S\cap[1,y]}\; \sum_{q\in \Q_{H,\nu}}\sum_{h\leq KH}\bl(H;q,n-a_{i,q}-qh) = \left(1+O\left(\frac{\log H}{H^{M-2}}\right)\right)|\Q_{H,\nu}|\cdot\lfloor KH\rfloor\s_1 y.
$$	
It is enough to show that, for any $h\leq KH$,
\begin{equation}\label{5.9}
\E\sum_{n\in \S\cap[1,y]}\sum_{q\in \Q_{H,\nu}}\bl(H;q,n-a_{i,q}-qh) = \left(1+O\left(\frac{\log H}{H^{M-2}}\right)\right)|\Q_{H,\nu}|\s_1 y.
\end{equation}
According to \eqref{lambda}, the left-hand side is equal to
$$
\E\sum_{n\in\S\cap[1,y]}\sum_{q\in \Q_{H,\nu}}\frac{\one_{\bAP(KH;q,n-a_{i,q}-qh)\subset \S_2}}{\s_2^{|\bAP(KH;q,n-a_{i,q}-qh)|}}
$$
By (\ref{2.11}), the condition $n\in \S\cap[1,y]$ implies that $n\in\S_1\cap[1,y]$. On the other hand, if $n\in\S_1$, then $n\in\bAP(KH;q,n-a_{i,q}-qh)$, and thus the condition $n\in\S_2$ is contained in the condition $\bAP(KH;q,n-a_{i,q}-qh)\subset\S_2$. So the left-hand side of (\ref{5.9}) can be rewritten as
$$
\E\sum_{n\in\S_1\cap[1,y]}\sum_{q\in \Q_{H,\nu}}\frac{\one_{\bAP(KH;q,n-a_{i,q}-qh)\subset \S_2}}{\s_2^{|\bAP(KH;q,n-a_{i,q}-qh)|}}.
$$
Recalling that $\S_2$ is independent of $\S_1$ and of $\bAP(KH;q,n-a_{i,q}-qh)$, we may apply Lemma \ref{lem5.1} as before and find that the left-hand side of (\ref{5.9}) is
$$
\E\sum_{n\in\S_1\cap[1,y]}\sum_{q\in \Q_{H,\nu}}\Bigg(1+O\Bigg(\frac{1}{H^{M-2}}+H^{-2}\sum_{a\in\A_q}\;\sum_{h',h''\leq KH}E_{A(H)}(a+qh'-qh'')\Bigg)\Bigg).
$$
Recall that $\E |\S_1 \cap [1,y]|=\sigma_1 y$ by Theorem \ref{th3} (i).
Thus, we see that (\ref{5.9}) follows from Lemma \ref{lem:sumE},
applied with $n,h',h''$ fixed, $w=0$, some choice of $a\in \cA_q$ for each $q$,
and $m_q=a+qh'-qh''$.
	
\medskip
	
Now we turn to the case $j=2$ of (iii), which states
\begin{multline}\label{iii:j=2}
\sum_{h_1,h_2\leq KH}\E\sum_{n\in\S\cap[1,y]}\;\sum_{q_1,q_2\in\Q_{H,\nu}}\bl(H;q_1,n-a_{i,q_1}-q_1h_1)\bl(H;q_2,n-a_{i,q_2}-q_2h_2)\\
=\left(1+O\left(\frac{1}{H^{M-2}}\right)\right)|\Q_{H,\nu}|^2\cdot\lfloor KH\rfloor^2\frac{\s_1}{\s_2}y.	
\end{multline}	
Arguing as in the $j=1$ case, the left-hand side equals
\begin{equation}\label{5.11}
\sum_{h_1,h_2\leq KH}\E\sum_{n\in\S_1\cap[1,y]}\;\sum_{q_1,q_2\in \Q_{H,\nu}}\frac{\one_{\bAP(KH;q_1,n-a_{i,q_1}-q_1h_1)\cup \bAP(KH;q_2,n-a_{i,q_2}-q_2h_2)\subset \S_2}}{\s_2^{|\bAP(KH;q_1,n-a_{i,q_1}-q_1h_1)|+|\bAP(KH;q_2,n-a_{i,q_2}-q_2h_2)|}};
\end{equation}
Note that here we again replace the condition $n\in\S\cap[1,y]$ by $n\in\S_1\cap[1,y]$ for the same reason as in $j=1$ case. Further, by \eqref{thm4(i)}, the contribution from $q_1=q_2$ is
$$ 
\ll H^2\s_2^{-2BKH}|\Q_{H,\nu}|\s_1y\ll |\Q_{H,\nu}|y^{1+o(1)}, 
$$
which, by \eqref{Q-asymp}, is an acceptable error term.

 We call a pair $(q_1,q_2)\in\Q_{H,\nu}^2$ with $q_1\neq q_2$ \textit{good}, if 
 for all $\S_1$, all $n\in \S_1 \cap [1,y]$ and all $h_1,h_2 \le KH$ we have
$$
\{n\} = \bAP(KH;q_1,n-a_{i,q_1}-q_1h_1)\cap\bAP(KH;q_2,n-a_{i,q_2}-q_2h_2),
$$ 
and call $(q_1,q_2)$ \textit{bad} otherwise;
recall that for any  $n\in \S_1 \cap [1,y]$, $n$ lies in both
$\bAP(KH;q_1,n-a_{i,q_1}-q_1h_1)$ and $\bAP(KH;q_2,n-a_{i,q_2}-q_2h_2)$.
 We need to estimate the number of bad pairs. First of all, if a pair $(q_1,q_2)$ is bad
 then there is a choice of $h_1,h_2$ so that both sets
$$
\bigsqcup_{j_1=1}^{\nu}\{a_{j_1,q_1}-a_{i,q_1}+q_1(h_1''-h_1): h_1''\leq KH\}  
$$ 
and
$$
\bigsqcup_{j_2=1}^{\nu}\{a_{j_2,q_2}-a_{i,q_2}+q_2(h_2''-h_2): h_2''\leq KH\}  
$$ 
contain the same nonzero number, say, $n_0$.
Fix $q_2$, $j_2$, $h_2$ and $h_2''$ so that 
\[
n_0 = a_{j_2,q_2}-a_{i,q_2}+q_2(h_2''-h_2).
\]
Then we have $n_0 \bmod q_1 \in I_{q_1}-I_{q_1}$.  By Hypothesis (f), the number of such
$q_1$ is $O(y^{0.49})$.  Therefore, the number of bad pairs $(q_1,q_2)$ is $ \ll y^{1.49}H^2 \ll y^{1.5}$.
 Since each of them contributes $y^{1+o(1)}$ to the left side of \eqref{5.11}, the contribution from these bad pairs is negligible.
	
It remains to estimate the contribution to \eqref{5.11} from good pairs
 $(q_1,q_2)$. Note that if $(q_1,q_2)$ is a good pair, then, for any 
 $\S_1, h_1,h_2,n$ the set
$$
\bAP(KH;q_1,n-a_{i,q_1}-q_1h_1)\cup \bAP(KH;q_2,n-a_{i,q_2}-q_2h_2)
$$
has size $|\bAP(KH;q_1,n-a_{i,q_1}-q_1h_1)|+|\bAP(KH;q_2,n-a_{i,q_2}-q_2h_2)|-1$. Then, as before, we can apply Lemma \ref{lem5.1} to rewrite the terms in (\ref{5.11}) corresponding to good $(q_1,q_2)$ as
\begin{multline}\label{E'E''}
\frac{\sigma y}{\sigma_2^2} \sum_{(q_1,q_2)\text{ good}}\; \sum_{h_1,h_2\le KH}\;\bigg(1+O\bigg(\frac{1}{H^{M-2}}\bigg) \bigg) + \\
+ O\Bigg( \frac{\sigma y}{\sigma_2^2 H^2} \sum_{q_1,q_2 \in \cQ_{H,\nu}}\; \bigg(
 H^2 E'(q_1)+ H^2 E'(q_2)+  E''(q_1,q_2)\bigg) \Bigg),
\end{multline}
where again we used that $\E |\S_1 \cap [1,y]| = \sigma_1 y$ from \eqref{thm4(i)},
that $\sigma_1/\sigma_2 = \sigma/\sigma_2^2$ and where we define
\[
E'(q)=\sum_{a\in \A_q}\; \sum_{h,h'\leq KH} E_{A(H)}(a+qh-qh';H)
\]
and
$$
E''(q_1,q_2)= \sum_{1\le h_1,h_2\le KH} \; \ssum{a_1 \in \cA_{q_1} \\ a_2\in \cA_{q_2}}\; \sum_{\substack{h_1',h_2'\leq KH}} E_{A(H)}(a_1-a_2+q_1h_1'-q_1h_1-q_2h_2'+q_2h_2;H).
$$

As the number of bad pairs $(q_1,q_2)$ is very small, the first line 
of \eqref{E'E''} produces the main term in \eqref{thm4(iii)} with an acceptable error.

By Lemma \ref{lem:sumE} with $w=0$, 
\be\label{E'}
\sum_{q_1,q_2\in \cQ_{H,\nu}} ( E'(q_1)+E'(q_2) ) \ll H^2 \, \frac{|\cQ_{H,\nu}|^2\log H}{H^{M-2}}.
\ee

For the sum on $E''(\cdot)$, if we have $a_1=a_2=h_1'-h_1=h_2'-h_2=0$
then the summand is zero for any $q_1,q_2$. 
Consider now the summands with either $a_1\ne 0$ or $h_1\ne h_1'$.
Fix $h_1,h_2,h_1',h_2',q_2,a_2$ and also a choice $a_1\in \cA_{q_1}$ for each
$q_1\in \cQ_{H,\nu}$.  Apply Lemma \ref{lem:sumE} to the sum over $q_1$,
with $w=-a_2-q_2 h_2' + q_2 h_2$ and $m_q = a_1 + q_1(h_1'-h_1)$ so that
$m_q\ne 0$.  A similar argument handles the case when $a_2\ne 0$ or $h_2\ne h_2'$, 
that is, fixing $q_1,a_1$ and summing over $q_2$, and we conclude that
\be\label{E''}
 \sum_{q_1,q_2 \in \cQ_{H,\nu}} \; E''(q_1,q_2) \ll H^4 \,  \frac{|\cQ_{H,\nu}|^2\log H}{H^{M-2}}.
\ee
Inserting \eqref{E'} and \eqref{E''} into \eqref{E'E''} establishes 
the desired bound \eqref{iii:j=2}.

 This completes the proof of the case $j=2$, and Theorem \ref{th3} (iii) follows.

\end{proof}

\end{document}